\documentclass[12pt]{article}
\usepackage{amsmath}
\usepackage{amsfonts}
\usepackage{amssymb}

\usepackage{amsthm} %\usepackage{thmdefs}
\usepackage{amscd}
\theoremstyle{plain} %% This is the default
\newtheorem{theorem}{Theorem}[section]
\newtheorem{corollary}[theorem]{Corollary}

\newtheorem{lemma}[theorem]{Lemma}

\newtheorem{definition}[theorem]{Definition}

\numberwithin{equation}{section}

\newcommand{\al}{\alpha}
\newcommand{\be}{\beta}

\newcommand{\bq}{\mathbf {Q}}
\newcommand{\bz}{\mathbf {Z}}

\newcommand{\ot}{\otimes}

\newcommand{\lra}{\longrightarrow}
\newcommand{\lla}{\longleftarrow}

\begin{document}
\title{Hochschild Cochains as a Frobenius Algebra}
\author{Jerry M. Lodder}
\date{}
\maketitle

\noindent
{\bf{Abstract.}} 
We construct a Frobenius algebra structure on the Hochschild cochains
of a group ring $k[G]$ that extends the known structure of a $\langle
1, \ 2 \rangle$ topological quantum field theory on $HH^0(k[G]; \,
k[G])$, $k$ a field and $G$ a finite group.  The convolution product
extends to the homotopy commutative Gerstenhaber product on cochains,
the Frobenius coproduct extends to a coproduct on the chain complex for
Hochschild homology, and there is a pairing $\langle \
,\ \rangle$ on Hochschild cochains satisfying 
$\langle \al \cdot \be , \ \gamma \rangle = \langle \al , \ \be \cdot
\gamma \rangle$.  The pairing, however, degenerates on a certain
subcomplex of the Hochschild cochains.



\section{Introduction}

In this paper we investigate how Hochschild cochains become a
differential graded homotopy commutative Frobenius algebra and thus a
target for a $\langle 1, \, 2 \rangle$ topological quantum field
theory \cite{Freed} from the category of oriented compact
one-manifolds with cobordisms as morphisms to the category of (co)chain
complexes and (co)chain maps (up to chain homotopy).  Since the direct
sum of the Hochschild cochains (over all dimensions) remains an
infinite dimensional vector space, the finite dimensionality axiom of a
TQFT is not satisfied.  
Specifically, let $k$ be a field and $G$ a finite group.  The
Hochschild cochains on $k[G]$ carry a product structure given by the
Gerstenhaber product, defined on Hochschild's original cochain complex
$$ \big( {\rm{Hom}}_k( k[G]^{\ot *}, \ k[G]), \ \delta \big).  $$
By the work of Tradler and Zeinalian \cite{TZ} and its generalizations 
\cite{Kaufmann, KS, WW}, the cochain complex 
${\rm{Hom}}_k( k[G]^{\ot *}, \ k[G])$ carries the structure of a PROP
over the (chains on) cyclic Sullivan diagrams.  This, however, does
not induce a $\langle 1, \, 2 \rangle$ TQFT on $HH^*(k[G]; \, k[G])$,
since a counit and pairing are missing.  Again, because 
$HH^*(k[G]; \, k[G])$ is not a finite dimensional vector space (summed
over all dimensions), a strict $\langle 1, \, 2 \rangle$ TQFT on 
$HH^*(k[G]; \, k[G])$ is not  possible.

The goal is this paper is to work instead with the $b^*$ cochain complex
$$  \big( {\rm{Hom}}_k( k[G]^{\ot (*+1)}, \ k), \ b^* \big).  $$ 
A duality pairing $k[G] \ot k[G] \to k$ for group rings induces a
cochain map from ${\rm{Hom}}_k( k[G]^{\ot *}, \ k[G])$ to
${\rm{Hom}}_k( k[G]^{\ot (*+1)}, \ k)$.  On the $b^*$ cochain complex,
the Gerstenhaber product becomes an extension of the convolution
product known to exist on the strict $\langle 1, \, 2 \rangle$ TQFT
offered by $HH^0(k[G]; \, k[G])$
\cite{Teleman}.  The main point of using the $b^*$ cochain complex, however,
is that there is a ``Gerstenhaber coproduct'' on the chain complex for
Hochschild homology, computed via the $b$-boundary map.  Let $C_n =
k[G]^{\ot (n+1)}$ and 
$C_* = \big( \sum_{n \geq 0}C_n, \ b \big)$ be the chain complex
for Hochschild homology.  We construct an explicit chain map
$$   T : C_* \to C_* \ot C_*  $$
that is dual to the Gerstenhaber product
$$  m: W_* \ot W_* \to W_*, \ \ \ m(\al \ot \be) := \al
\underset{G}{\cdot} \be,  $$
where $W_n = {\rm{Hom}}_k(C_n, \ k)$.  Moreover, we define a pairing
$\langle \ \, \ \rangle : W_* \ot W_* \to k$ so that
$$  \langle \al \underset{G}{\cdot} \be , \ \gamma \rangle = 
\langle \al , \ \be \underset{G}{\cdot} \gamma \rangle, $$
which establishes that $W_*$ is a Frobenius algebra,
commutative up to homotopy.  The pairing, however, fails to be
non-degenerate, and we establish a subcomplex on which $\langle \, ,
\, \rangle$ degenerates, namely 
$$  V_p = \{ \al \in W_p \ | \ \al (h, \, N, \, \ldots \, , \, N) = 0
\ \forall h \in G \}, \ p \geq 0,  $$
where $N = \sum_{g \in G} g$ is the norm element of $k[G]$.  

Behind the Hochschild homology of the group ring $k[G]$ is the
simplicial set $N^{\rm{cy}}_*(G)$, the cyclic bar construction of $G$
\cite[7.3.10]{cyclic-hom}, with geometric realization
$$  |N^{\rm{cy}}_*(G)| \simeq {\rm{Maps}}(S^1, \ BG),  $$
where $S^1$ denotes the unit circle and $BG$ is the classifying space
of $G$.  Recall that $N^{\rm{cy}}_n(G) = G^{n+1}$ and there is a
subsimplicial set of $N^{\rm{cy}}_*(G)$ given by
$$  N^{\rm{cy}}_n(G,\, e) = \{ (g_0, \, g_1, \, \ldots \, , \, g_n)
\in G^{n+1} \ | \ g_0g_1 \ldots g_n = e \}  $$ 
with $|N^{\rm{cy}}_*(G, \, e)| \simeq BG$, which realizes constant
maps of $S^1$ into $BG$ (the $S^1$-fixed points of ${\rm{Maps}}(S^1, \
BG)$).  Within the Hochschild complex
$W_n = {\rm{Hom}}_k(k[G]^{\ot (n+1)}, \, k)$ we define cochains
supported on $BG$ as elements $\al \in W_n$ with 
$\al(g_0, \, g_1, \, \ldots \, , \, g_n) = 0$ whenever $g_0g_1 \ldots
g_n \neq e$.  These cochains form a differential graded subcomplex of
$W_*$, closed under the Gerstenhaber product, and thus a subalgebra.
Moreover, for cochains supported on $BG$, the Gerstenhaber product
agrees with the simplicial cup product, and Gerstenhaber's pre-Lie
product agrees with Steenrod's cup-one product. This establishes that the
cochain complex for group cohomology, under the simplicial cup product,
is a homotopy commutative subalgebra of $W_*$ under the Gerstenhaber
product.  

Recall that $b: k[G]^{\ot 2} \to k[G]$ is given by the commutator
$b(g_0, \, g_1) = g_0g_1 - g_1g_0$, and
$$  H = {\rm{Ker}} \, b^* : {\rm{Hom}}_k(k[G], \ k) \to
{\rm{Hom}}_k( k[G]^{\ot 2}, \ k)  $$
carries the structure of a $\langle 1, \ 2 \rangle$ TQFT, at least
when $G$ is finite and $k$ a field \cite{Teleman}.  For $\al$, $\be
\in H$, the product is given by the convolution product $\al * 
\be$, 
$$  (\al * \be)(g_0) = \sum_{h \in G} \al (h) \be(g_0 h^{-1}).  $$
The coproduct $T: H \to H \ot H$ is induced by
$$  T(g_0^*) =  \sum_{h \in G} h^* \ot (g_0h^{-1})^* , $$
where $g^*$ is the element of ${\rm{Hom}}_k(k[G],\ k)$ dual to $g
\in G$.  Using $V = HH_0(k[G]; \, k[G])$ and $V^* \simeq H$, 
the coproduct $T : V \to V \ot V$ can also be written as
$$  T(g_0) = \sum_{h \in G} h \ot g_0 h^{-1}.  $$
The duality pairing $\langle \ , \ \rangle : H \ot H \to k$ is 
given by
$$  \langle \al ,\, \be \rangle = \sum_{h \in G} \al(h)
\be(h^{-1}).  $$
We show how each of the above, the convolution product, the coproduct
and the pairing can be extended to the full Hochschild cochain
complex.  Although the resulting Frobenius algebra structure is stated
for finite groups, many constructions work for infinite groups and we
work with an arbitrary discrete group where possible.

\section{Hochschild Cohomology}

Let $k$ be a unital, commutative and associative coefficient ring and
let $A$ be an associative algebra over $k$.  Recall that Hochschild's
original definition \cite{Hochschild1944, Hochschild1945} for 
$HH^*(A; \, A)$, the Hochschild cohomology of $A$ with coefficients in
the bimodule $A$ is given as the homology of the cochain complex:  
\begin{align*} 
& {\rm{Hom}}_k(k, \, A) \overset{\delta}{\lra} {\rm{Hom}}_k(A, \, A)
\overset{\delta}{\lra} \ldots \\
& \ldots  \overset{\delta}{\lra} 
{\rm{Hom}}_k(A^{\ot n}, \, A) \overset{\delta}{\lra} 
{\rm{Hom}}_k(A^{\ot (n+1)}, \, A) \overset{\delta}{\lra}  \ldots \, .
\end{align*}
For a $k$-linear map $f: A^{\ot n} \to A$, the coboundary 
$\delta f : A^{\ot (n+1)} \to A$ is given by
\begin{align*}
& (\delta f)(a_1, \, a_2, \, \ldots \, , a_{n+1}) = a_1 f(a_2, \, \ldots
\, , a_{n+1})  \, + \\
& \Big( \sum_{i=1}^n (-1)^i f(a_1, \, a_2, \, \ldots \, , a_i a_{i+1}, \,
\ldots \, , a_n) \Big)  + (-1)^{n+1} f(a_1, \, a_2, \, \ldots \, , a_n)
a_{n+1}.  
\end{align*}
Of course, for $n = 0$, $(\delta f)(a_1) = a_1 f(1) - f(1)a_1$.  For 
$f \in {\rm{Hom}}_k(A^{\ot p}, \, A)$ and $g \in {\rm{Hom}}_k(A^{\ot
  q}, \, A)$, the Gerstenhaber (cup) product   
\cite{Gerstenhaber} 
$$  f \underset{G}{\cdot} g \in {\rm{Hom}}_k(A^{\ot (p+q)}, \, A) $$
is given by 
$$  (f \underset{G}{\cdot} g) (a_1, \, a_2, \, \ldots \, , \, a_{p+q})
=  f(a_1, \, \ldots \, , \, a_p) \cdot g(a_{p+1}, \, \ldots \, , \,
a_{p+q}).  $$
Gerstenhaber proves that $f \underset{G}{\cdot} g$ induces a graded
commutative product on $HH^*(A; \, A)$.  If $f \in HH^p(A; \, A)$ and 
$g \in HH^q(A; \, A)$, then $f \underset{G}{\cdot} g \in HH^{p+q}(A;
\, A)$.  

On the other hand, the Hochschild homology groups \cite[X.4]{Maclane}
$HH_*(A; \, A)$ are given by the homology of the chain complex
$$  A \overset{b}{\lla} A^{\ot 2} \overset{b}{\lla} \, \ldots \, 
\overset{b}{\lla} A^{\ot n} \overset{b}{\lla} A^{\ot (n+1)}
\overset{b}{\lla} \, \ldots \, ,  $$
where for $(a_0, \, a_1, \, \ldots \, , a_n) \in A^{\ot (n+1)}$,
\begin{align*}
&  b(a_0, \, a_1, \, \ldots \, , a_n) = \\
& \Big( \sum_{i=0}^{n-1} (-1)^i (a_0, \, \ldots \, , a_i a_{i+1}, \,
\ldots \, , a_n) \Big) + (-1)^n (a_n a_0,, a_1, \, \ldots \, , a_n).
\end{align*}
By $(HH^*(A), \, b^*)$ we mean the homology of the ${\rm{Hom}}_k$-dual
of the $b$-complex, given by 
\begin{align*} 
&{\rm{Hom}}_k(A, \, k) \overset{b^*}{\lra} {\rm{Hom}}_k(A^{\ot 2}, \, k)
\overset{b^*}{\lra} \, \ldots \\ 
&\ldots \, \overset{b^*}{\lra}
{\rm{Hom}}_k(A^{\ot n}, \, k) \overset{b^*}{\lra} 
{\rm{Hom}}_k(A^{\ot (n+1)}, \, k) \overset{b^*}{\lra} \, \ldots \, .  
\end{align*} 
For a $k$-linear map $\varphi : A^{\ot n} \to k$, $b^*(\varphi)
: A^{\ot (n+1)} \to k$ is given by
$$  b^*(\varphi) (a_0, \, \ldots \, , \, a_n) = \varphi(b(a_0, \,
\ldots \, , \, a_n)).  $$

Let $A = k[G]$ be a group ring.  For $g$, $h \in G$, define a
symmetric bilinear form
$\langle \ , \ \rangle : k[G] \times k[G] \to k$ with
$$  \langle g, \, h \rangle = \begin{cases} 1, & h = g^{-1} \\
                                            0, & h \neq g^{-1}.
                              \end{cases}  $$
Then extend $ \langle \ , \ \rangle $ to be linear in each factor,
which results in a $k$-linear map on the tensor product
$\langle \ , \ \rangle : k[G] \ot k[G] \to k$.  There is an injective
cochain map 
$$  \Phi_n : \big( {\rm{Hom}}_k (k[G]^{\ot n} , \, \, k[G]), \ \delta
\big) \to \big( {\rm{Hom}}_k 
 (k[G]^{\ot (n+1)} , \, \, k), \ b^* \big),  \ \ \   n \geq 0, $$
given by 
$$  \Phi_n (f) (g_0, \, g_1, \, g_2, \, \ldots \, , \, g_n) =
\langle g_0, \ f(g_1, \, g_2, \, \ldots \, , \, g_n) \rangle . $$
When $G$ is finite, $\Phi_n$ is, of course, an isomorphism of cochain complexes.

For $G$ an arbitrary group, we introduce a particular formula for
$\Phi_n$ used throughout the paper.  Recall that for $k$ an arbitrary
unital, commutative coefficient ring, $k[G]$ is a free $k$-module with
basis given by the elements of $G$.  For $g_0$,
$g_1$, $\ldots \,$, $g_n \in G$ and $h_1$, $h_2$, $\ldots \, $, $h_n
\in G$, let
$$  (g_0, \, g_1, \, \ldots \, , \, g_n)^{\#} : k[G]^{\ot n} \to k[G]  $$
denote the $k$-linear map determined by 
$$  (g_0, \, g_1, \, \ldots \, , \, g_n)^{\#} (h_1, \, h_2, \, \ldots
\, , \, h_n) = \begin{cases} g_0, & h_1 = g_1, \, \ldots \, , \, h_n =
                                                              g_n , \\
                             0,    & {\rm{otherwise}}.  \end{cases}  $$
Additionally, for $h_0 \in G$, let $(g_0, \, g_1, \, \ldots \, , \,
g_n)^{*}: k[G]^{\ot (n+1)} \to k$ be the $k$-linear map determined
by
$$  (g_0, \, g_1, \, \ldots \, , \, g_n)^{*} (h_0, \, h_1, \, \ldots
\, , \, h_n) = \begin{cases} 1, & 
h_0 = g_0, \,  h_1 = g_1, \, \ldots \, , \, h_n = g_n, \\
                             0,  & {\rm{otherwise}}.  \end{cases}  $$
Under this notation,
$$  \Phi_n \big( (g_0, \, g_1, \, g_2, \, \ldots \, , \, g_n)^{\#} \big) =
(g^{-1}_0, \, g_1, \, g_2, \, \ldots \, , \, g_n)^* .  $$

Let $I_n = {\rm{Im}} \, \Phi_n \subseteq {\rm{Hom}}_k(k[G]^{\ot (n+1)},
\, k)$.  Then $I_* = \{ I_n \}_{n \geq 0}$ is a subcomplex of 
${\rm{Hom}}_k(k[G]^{\ot (*+1)}, \, k)$, and there is a cochain map
\cite[Lemma 2.3]{Lodder} 
$$  \Psi_n : I_n \to {\rm{Hom}}_k (k[G]^{\ot n}, \, k[G]), \ \ \  n
\geq 0,  $$
induced by 
\begin{equation} \label{Psi}
 \Psi_n \big( (g_0, \, g_1, \, g_2, \, \ldots \, , \, g_n)^* \big) =
(g^{-1}_0, \, g_1, \, g_2, \, \ldots , \, , g_n)^{\#} .  
\end{equation} 
Then $\Psi_n$ is extended via $k$-linearity to all of $I_n$.  Clearly 
$ \Psi_n \circ \Phi_n = \bf{1}$ on ${\rm{Hom}}_k (k[G]^{\ot n}, \,
k[G])$, showing that $\Phi_n$ is injective on cochains for an
arbitrary discrete group $G$.    

\begin{lemma} \label{injectivity}
For any discrete group $G$, $k$ a field,  the induced map
$$ \Phi^* : HH^*(k[G]; \, k[G]) \to \big( HH^*(k[G]), \, b^* \big)  $$
is injective.
\end{lemma}
\begin{proof}
We first borrow some ideas from the universal coefficient
theorem, which states that for $k$ a field, there is a natural
isomorphism
$$  \big( HH^*(k[G]), \, b^* \big) \to {\rm{Hom}}_k(HH_*(k[G]; \,
  k[G]), \, k).  $$
Let $C_* = \big( k[G]^{\ot (*+1)}, \, b \big)$ be the chain complex for
Hochschild homology.  As vector spaces, $C_* \simeq A_* \oplus (B_*
\oplus E_*)$, where $B_*$ is the subspace of boundaries of $C_*$, 
$A_* \simeq Z_*/B_*$, $Z_*$ is the subspace of cycles of $C_*$, and
$E_*$ is a complementary subspace.  As chain complexes,
$$  (C_*, \, b) \simeq (A_*, \, 0) \oplus (B_* \oplus E_*, \, b),  $$
where the boundary map of $A_*$ is 0.  Clearly,
$ A_* = H_*(A_*) \simeq HH_*(k[G]; \, k[G])$.  
As cochain complexes,
$$  \big( {\rm{Hom}}_k (C_*, \, k), \, b^* \big) \simeq
\big( {\rm{Hom}}_k (A_*, \, k), \, 0 \big) \oplus
\big( {\rm{Hom}}_k (B_* \oplus E_*, \, k), \, b^* \big).  $$
On cohomology ${\rm{Im}}\Phi^*$ is contained in
$$  {\rm{Hom}}_k (A_*, \, k) = H^*({\rm{Hom}}_k (A_*, \, k)) \simeq
\big( HH^* (k[G]), \, b^* \big).  $$

Let $J_n$ be the subspace of ${\rm{Hom}}_k (A_n , \, k)$ given by 
$\al : A_n \to k$ with $\al$ finitely supported on the first tensor
factor of $k[G]^{\ot (n+1)}$, i.e.,  for any $x_i \in G$,
$$  \al( h, \, x_1, \, x_2 , \, \ldots \, , \, x_n) \neq 0  $$
for only finitely many $h \in G$.  We claim that 
${\rm{Im}}(\Phi^*_n) = J_n$ on cohomology.  First, let
$f \in {\rm{Hom}}_k (k[G]^{\ot n}, \, k[G])$.  Then 
$f(x_1, \, x_2, \, \ldots \, , \, x_n) = \sum_{i=1}^m c_i h_i$,
where each $c_i$ is a non-zero element of $k$.  
Thus, $\Phi_n(f)(h, \, x_1, \, \ldots \, , \, x_n) \neq 0$ only for $h
= h_i^{-1}$, $i = 1, \, 2, \, \ldots \, , m$, and
${\rm{Im}}(\Phi^*_n) \subseteq J_n$. Second, let $\al \in J_n$.  
Then $\al (h, \, x_1, \, \ldots \, , \, x_n) \neq 0$ only for 
finitely many $h = h_1, \, h_2, \, \ldots \, , \, h_p$, and
$\Psi_n (\al) \in {\rm{Hom}}_k (k[G]^{\ot n}, \, k[G])$, where $\Psi$
is the cochain map given via equation \eqref{Psi}.  Thus,
$\Phi_n (\Psi_n (\al)) = \al$, and $J_n \subseteq {\rm{Im}}
(\Phi^*_n)$.  

As vector spaces, ${\rm{Hom}}_k(A_n, \, k) \simeq J_n \oplus M_n$,
where $M_n$ is a complementary subspace.  
Finally, $\Psi$ extends to a cochain map
$$ \Psi_n : \big( {\rm{Hom}}_k (A_n , \, k), \, 0 \big) \to
\big( {\rm{Hom}}_k ( k[G]^{\ot n}, \, k[G]), \, \delta \big) $$
by setting $\Psi (\al ) = 0$ for $\al \in M_n$.  If $f \in
{\rm{Hom}}_k( k[G]^{\ot n}, \, k[G])$ represents a cohomology class in
$HH^n(k[G]; \, k[G])$, then using the explicit formulas for $\Phi_n$
and $\Psi_n$, we have
$$  \Psi^*_n \circ \Phi^*_n ([f]) = [\Psi_n \circ \Phi_n (f)] = [f].  $$
Thus, $\Psi^* \circ \Phi^* = {\bf{1}}$, proving the injectivity of
$\Phi^*$.  Since the injectivity of $\Phi^*$ holds for an arbitrary 
field $k$, the result also holds for $k = \bz$.
\end{proof}

The above lemma suggests the study of $b^*$ cohomology, $\big(
HH^*(k[G]), \, b^* \big)$, as a target for a TQFT, since (i) $b^*$
cohomology contains an injective image of $HH^*(k[G]; \, k[G])$, and
(ii) $b^*$ cohomology is more closely related to the free loop space
and string theory via 
$$ \big( HH^*(k[G]), \, b^* \big) 
\simeq H^*({\rm{Maps}}(S^1, \, BG)), $$
$k$ any commutative, unital coefficient ring.

\begin{corollary}
Let $G$ be an arbitrary discrete group, $k$ a field.  If
${\rm{Maps}}(S^1, \, BG)$ is of finite type, i.e., $HH_n(k[G]; \,
k[G])$ is a finite dimensional vector space for each $n$, then the
induced map 
$$ \Phi^* : HH^*(k[G]; \, k[G]) \to \big( HH^*(k[G]), \, b^* \big)  $$ 
is an isomorphism.  
\end{corollary}
\begin{proof}
Using the notation of Lemma \eqref{injectivity} , we see that
${\rm{Hom}}_k(A_n, \, k)$ is a finite dimensional vector space for
each $n$.  Thus, $\Psi^* : {\rm{Hom}}_k(A_n, \, k) \to
HH^*(k[G]; \, k[G])$ is an isomorphism.  Since the result holds for an
arbitrary field $k$, $\Phi^*$ is also an isomorphism for $k = \bz$
when ${\rm{Maps}}(S^1, \, BG)$ is of finite type.
\end{proof}

We now define the Gerstenhaber product on the
complex $I_*$.  

\begin{definition}  \label{b*Gerstenhaber}
Let 
$$ \al \in I_p \subseteq {\rm{Hom}}_k(k[G]^{\ot (p+1)}, \, k), \ \ \ 
\be \in I_q \subseteq {\rm{Hom}}_k(k[G]^{\ot (q+1)}, \, k).  $$  
Then
$$  \al \underset{G}{\cdot} \be \in I_{p+q} \subseteq
{\rm{Hom}}_k(k[G]^{\ot (p+q+1)}, \, k) $$
is defined by
$$  \al \underset{G}{\cdot} \be = \Phi_{p+q} \big( \Psi_p (\al)
\underset{G}{\cdot} \Psi_q( \be) \big).  $$
\end{definition}

\begin{lemma}
The Gerstenhaber product is well-defined on $(HH^*(I_*), \ b^*)$, the
Hochschild cohomology of the $I_*$ complex.
\end{lemma}
\begin{proof}
With $\al$ and $\be$ as in Definition \eqref{b*Gerstenhaber}, 
\begin{align*}
& b^*( \al \underset{G}{\cdot} \be ) = \Phi_{p+q+1} \big( \delta (
\Psi_p(\al) \underset{G}{\cdot} \Psi_q (\be)) \big) \\
& = \Phi_{p+q+1} \big( \delta( \Psi_p( \al )) \underset{G}{\cdot}
\Psi_q (\be ) + (-1)^p \Psi_p ( \al ) \underset{G}{\cdot}
\delta ( \Psi_q (\be) ) \big) \\
& = \Phi_{p+q+1} \big(  \Psi_p(b^*( \al )) \underset{G}{\cdot}
\Psi_q (\be ) + (-1)^p \Psi_p ( \al ) \underset{G}{\cdot}
\Psi_q ( b^* (\be) ) \big).
\end{align*}
\end{proof}

\begin{lemma} \label{Gerstenhaberproduct}
Let 
\begin{align*}
& \al = (\al_0, \ \al_1, \ \ldots \, , \ \al_p)^* \in I_p, \ \ \ \al_i
\in G, \\
& \be = (\be_0, \ \be_1, \ \ldots \, , \ \be_q)^* \in I_q, \ \ \ \be_i
\in G.  
\end{align*}
Then
$$ \al \underset{G}{\cdot} \be =
( \be_0 \al_0, \ \al_1, \ \ldots \, , \ \al_p, \ \be_1, \ \ldots \, , \
\be_q)^* .  $$
\end{lemma}
\begin{proof}
\begin{align*}
& \al \underset{G}{\cdot} \be =
\Psi_{p+q}( (\al_0^{-1}, \, \al_1, \, \ldots \, , \, \al_p)^{\#}
\underset{G}{\cdot}
(\be_0^{-1}, \, \be_1, \, \ldots \, , \, \be_q)^{\#}) \\
& = \Psi_{p+q}( (\al_0^{-1} \be_0^{-1}, 
\ \al_1, \, \ldots \, , \, \al_p, \ \be_1, \, \ldots \, , \be_q)^{\#}
\\
& = ( \be_0 \al_0, \ \al_1, \ \ldots \, , \ \al_p, \ \be_1, \ \ldots
\, , \
\be_q)^* 
\end{align*}
\end{proof}
The above lemma express the Gerstenhaber product on natural $k$-module
generators of $I_*$, which by linearity becomes the
convolution product on $I_*$, proven below.  

\begin{lemma}  \label{G-product}
Let $\al \in I_p$ and $\be \in I_q$ be arbitrary. Then
\begin{align*}
& ( \al \underset{G}{\cdot} \be )(g_0, \, g_1, \, \ldots \, , \, g_p ,
\, g_{p+1}, \, \ldots \, , \, g_{p+q}) \\
& = \sum_{h \in G} \al(h, \, g_1, \, \ldots \, , \, g_p) 
\be(g_0 h^{-1}, \, g_{p+1}, \, \, \ldots \, , \, g_{p+q}).
\end{align*}
\end{lemma}
\begin{proof}
In the special case $\al = (\al_0, \, \al_1, \, \ldots \, , \,
\al_p)^*$ and $\be = (\be_0, \, \be_1, \, \ldots \, , \, \be_q)^*$, it
follows from 
$$  \al \underset{G}{\cdot} \be =
( \be_0 \al_0, \ \al_1, \ \ldots \, , \ \al_p, \ \be_1, \ \ldots \, , \
\be_q)^*  $$
that $(\al \underset{G}{\cdot} \be)(g_0, \, g_1, \, \ldots \, , \,
g_{p+q})$ is non-zero only if $g_0 = \be_0 \al_0$.  Also, the only
non-zero summand in 
$$  \sum_{h \in G} \al(h, \, g_1, \, \ldots \, , \, g_p) 
\be(g_0 h^{-1}, \, g_{p+1},  \, \ldots \, , \, g_{p+q})  $$
occurs when $h = \al_0$ and $g_0 = \be_0 \al_0$, in which case
$$  (\al \underset{G}{\cdot} \be)(g_0, \, \ldots \, , \,
g_{p+q}) = \sum_{h \in G} \al(h, \, g_1, \, \ldots \, , \, g_p) 
\be(g_0 h^{-1}, \, g_{p+1}, \, \ldots \, , \, g_{p+q}). $$
The lemma follows by linearity, with the details below. 

For arbitrary $\al \in I_p$, $\be \in I_q$, 
let $f_1 = \Psi_p (\al)$, $f_2 = \Psi_q (\be)$.  Then
$f_1(g_1, \, \ldots \, , \, g_p) = \sum_{i=1}^n c_i \mu_i$, where $c_i
\in k$ and $\mu_i$ are distinct\ elements of $G$.
Thus, for fixed $g_1$, $g_2$, $\ldots$, $g_p \in G$, 
$\al(h, \, g_1, \, \ldots \, , \, g_p)$ is non-zero on only finitely
many $h \in G$.  A similar statement holds for
$\be(g_0 h^{-1}, \, g_{p+1}, \, \, \ldots \, , \, g_{p+q})$.  Thus the
sum
$$  \sum_{h \in G} \al(h, \, g_1, \, \ldots \, , \, g_p) 
\be(g_0 h^{-1}, \, g_{p+1}, \, \, \ldots \, , \, g_{p+q}) $$
is well-defined, even when $G$ is an infinite group.  Consider
\begin{align*}
&  \al = \sum_{i=1}^n c_i(a_i, \, g_1, \, \ldots \, , \, g_p)^*, \ \
a_i \in G, \ c_i \in k, \\
& \be = \sum_{i=1}^m \ell_j (b_j, \, g_{p+1}, \, \ldots \, , \,
g_{p+q})^*, \ \ b_j \in G, \ \ell_j \in k .
\end{align*}
From Lemma \eqref{Gerstenhaberproduct}
$$  (\al \underset{G}{\cdot} \be) = \sum_{i, \, j} 
c_i \ell_j (b_j a_i, \, g_1, \, \ldots \, , \, g_{p+q})^*.  $$
The only non-zero terms of $(\al \underset{G}{\cdot} \be)^*(g_0, \,
\ldots \, , \, g_{p+q})$ occur when $g_0 = b_j a_i$ for some $i$ and
$j$.  Likewise the only non-zero summands of 
$$  \sum_{h \in G} \al(h, \, g_1, \, \ldots \, , \, g_p) 
\be(g_0 h^{-1}, \, g_{p+1}, \, \, \ldots \, , \, g_{p+q}) $$
occur when $h = a_i$ and $g_0 = b_j a_i$, in which case
$$  (\al \underset{G}{\cdot} \be)(g_0, \, \ldots \, , \,
g_{p+q}) = \sum_{h \in G} \al(h, \, g_1, \, \ldots \, , \, g_p) 
\be(g_0 h^{-1}, \, g_{p+1}, \, \ldots \, , \, g_{p+q}). $$   
\end{proof}
\begin{definition}
An element $\al \in {\rm{Hom}}_k (k[G], \, k)$ with $\al (h) \neq 0$
for only finitely many $h \in G$ is called finitely supported.
\end{definition}
Note that for $\al \in I_0 \subseteq {\rm{Hom}}_k (k[G], \, k)$, $\al$ is
finitely supported.
\begin{corollary}
For $\al$, $\be \in {\rm{Hom}}_k (k[G], \, k)$ finitely supported, the
Gerstenhaber product is given by the convolution product, i.e.,
$$ \al \underset{G}{\cdot} \be = \al * \be =
\sum_{h \in G} \al(h) \be(g_0h^{-1}).  $$
\end{corollary}
The unit $u \in {\rm{Hom}}_k (k[G], \, k)$ for the Gerstenhaber
product on $I_q$, $q \geq 0$, is given by
$$  u(h) = \begin{cases} 1, \ \ \ h = e, \ \ h \in G, \\
                         0, \ \ \ h \neq e, \ \ h \in G.
\end{cases}  $$

\section{The Gerstenhaber Coproduct}

In this section we show that the Frobenius coproduct on a group ring
can be extended to a coproduct on the chain complex for Hochschild
homology that becomes essentially a ``Gerstenhaber coproduct.'' 
The results are most easily stated for a finite group
$G$, although formally can be extended to infinite groups by using
completed group rings, completed tensor products and 
coefficients
in a $p$-adic completion $\hat{\bq}_p$.  For $G$ finite, let
$$  T: k[G] \to k[G] \ot k[G]  $$
be the $k$-linear map with $T(g_0) = \sum_{h \in G} h \ot g_0h^{-1}$,
where $g_0 \in G$.  Then $T$ is often called the Frobenius coproduct.
We refrain from using $\Delta$ for the Gerstenhaber coproduct, since
$\Delta$ denotes the simplicial coproduct.  
Let $C_n = k[G]^{\ot (n+1)}$, $n \geq 0$, be the chain complex for
$HH_*(k[G]; \, k[G])$.  Then $C_* = k[N^{\rm{cy}}_*(G)]$, where
$N^{\rm{cy}}_*(G)$ denotes the cyclic bar construction on $G$ 
\cite[7.3.10]{cyclic-hom}.  

Let $C_* \ot C_*$ denote the tensor product of chain complexes, i.e.,
$$  (C_* \ot C_*)_m = \sum_{p=0}^m C_p \ot C_{m-p}, \ \ \ m \geq 0,  $$
with differential $b^{\rm{Tot}} = b \ot {\bf{1}} + (-1)^p {\bf{1}} \ot b$,
$$  b^{\rm{Tot}} : C_p \ot C_{m-p} \to (C_{p-1} \ot C_{m-p}) \oplus
(C_p \ot C_{m-p-1}).  $$
It is an interesting exercise to show that $T: k[G] \to k[G] \ot k[G]$
described above is the beginning of a chain map
$T: C_* \to C_* \ot C_*$.  
\begin{theorem}
The $k$-linear map
\begin{align*}
& T: C_* \to C_* \ot C_*, \\
& T: C_m \to \sum_{p=0}^m C_p \ot C_{m-p}, \\
& T(g_0, \, g_1, \, \ldots \, , \, g_m) =
\sum_{p=0}^m \sum_{h \in G} (h, \, g_1, \, \ldots \, , \, g_p) \ot 
(g_0 h^{-1}, \, g_{p+1}, \, \ldots \, , \, g_m)
\end{align*}
is a map of chain complexes, i.e., $T \circ b = b^{\rm{Tot}} \circ T$.
\end{theorem}
\begin{proof}
Recall that 
$$  W_n := {\rm{Hom}}_k(k[G]^{\ot (n+1)}, \ k) = {\rm{Hom}}_k(C_n, \,
k).  $$
Since $m : W_* \ot W_* \to W_*$ given by $m ( \al \ot \be ) =
\al \underset{G}{\cdot} \be$ is a map of cochain complexes, there is a
map of chain complexes
$$  m^* : {\rm{Hom}}_k(W_*, \, k) \to 
{\rm{Hom}}_k( W_* \ot W_*, \, k).  $$
For $G$ finite, there is a natural vector space isomorphism
${\rm{Hom}}_k(W_n, \, k) \simeq k[G]^{\ot (n+1)}$ given by
$$  (g_0, \, g_1, \, \ldots \, , \, g_n)^{**} 
\longleftrightarrow  (g_0, \, g_1, \, \ldots \, , \, g_n).  $$
By construction of $m^*$, we have
\begin{align*}
& m^*( (g_0, \, g_1, \, \ldots \, , \, g_n)^{**}) (\al \ot \be) \\
& = (g_0, \, g_1, \, \ldots \, , \, g_n)^{**}( \al \underset{G}{\cdot}
\be) \\
& = (\al \underset{G}{\cdot} \be)(g_0, \, g_1, \, \ldots \, , \, g_n) \\
& = \sum_{h \in G} \al(h, \, g_1, \, \ldots \, , \, g_p) 
\be(g_0 h^{-1}, \, g_{p+1}, \, \, \ldots \, , \, g_{n}).
\end{align*}
Also,
\begin{align*}
&\Big( \sum_{h \in G} (h, \, g_1, \, \ldots \, , \, g_p)^{**}
\ot (g_0h^{-1}, \, g_{p+1}, \, \ldots \, , \, g_n)^{**} \Big)
(\al \ot \be) \\
& = \sum_{h \in G} (h, \, g_1, \, \ldots \, , \, g_p)^{**}( \al)
(g_0h^{-1}, \, g_{p+1}, \, \ldots \, , \, g_n)^{**}(\be) \\
& = \sum_{h \in G} \al (h, \, g_1, \, \ldots \, , \, g_p) 
\be (g_0h^{-1}, \, g_{p+1}, \, \ldots \, , \, g_n).  
\end{align*}
Thus,
$$   m^*( (g_0, \, g_1, \, \ldots \, , \, g_n)^{**}) =
\sum_{p=0}^m \sum_{h \in G} (h, \, g_1, \, \ldots \, , \, g_p)^{**} \ot
(g_0h^{-1}, \, g_{p+1}, \, \ldots \, , \, g_n)^{**}.  $$
Under the double hom-dual functor on vector spaces and linear maps,
the morphism $m^*$ corresponds to the morphism $T$ given in the statement of
the theorem.
\end{proof}

Thus, there is an induced map
$T_* : HH_*(k[G]; \, k[G]) \to H_*( C_* \ot C_* )$.  For $k$ a field,
we have
$$ T_* : HH_*(k[G]; \, k[G]) \to  HH_*(k[G]; \, k[G]) \ot
 HH_*(k[G]; \, k[G]).  $$
The counit (trace) $\tau : k[G]^{\ot n+1} \to k$ is given by $\tau
(\sigma) = 0$, $n \geq 1$, and for $n = 0$,
$$  \tau \big( \sum_{h \in G} c_h \, h \big) = c_e ,  $$
where $c_e$ is the coefficient on the identity element $e \in G$.
Then
$$  ( \tau \ot {\bf{1}}) ( T( \sigma)) = \sigma, \ \ \
\sigma \in k[G]^{\ot (p+1)}.  $$
As usual, for $\al \in {\rm{Hom}}_k (k[G]^{\ot (p+1)}, \, k)$
and $\be \in {\rm{Hom}}_k (k[G]^{\ot (q+1)}, \, k)$, 
$\al \ot \be \in {\rm{Hom}}( C_* \ot C_*, \, k)$ is defined by
$$  (\al \ot \be)(\sigma_1 \ot \sigma_2) = \al (\sigma_1) \be
(\sigma_2), $$
where $\al (\sigma_1) = 0$ if $\sigma_1 \in k[G]^{\ot (n +1)}$, $n
\neq p$, and $\be (\sigma_2) = 0$ if $\sigma_2 \in k[G]^{\ot (m+1)}$, $m
\neq q$.  Then
\begin{equation} \label{coproduct}
( \al \underset{G}{\cdot} \be ) ( \sigma) = (\al \ot \be)(T(\sigma)),
\ \ \ \sigma \in k[G]^{\ot (p+q+1)}.  
\end{equation}
In the sense of equation \eqref{coproduct}, $T$ is a coproduct for the
Gerstenhaber product.

\section{A Frobenius Algebra on Cochains}
In this section we work with a finite group $G$ and a field $k$.  We
claim that the Hochschild cochains $({\rm{Hom}}(k[G]^{\ot (*+1)}, \,
b^*)$ under the Gerstenhaber (convolution) product form a differential
graded homotopy commutative Frobenius algebra.  Let $W_n =
{\rm{Hom}}(k[G]^{\ot (n+1)}, \, k)$ and $W_* = \sum_{n \geq 0}W_n$.  
For $\al \in W_p$ and $\be \in
W_q$, recall from Lemma \eqref{G-product} that
\begin{align*}
& ( \al \underset{G}{\cdot} \be )(g_0, \, g_1, \, \ldots \, , \, g_p ,
\, g_{p+1}, \, \ldots \, , \, g_{p+q}) \\
& = \sum_{h \in G} \al(h, \, g_1, \, \ldots \, , \, g_p) 
\be(g_0 h^{-1}, \, g_{p+1}, \, \, \ldots \, , \, g_{p+q}).
\end{align*}
The chain homotopy between $\al \underset{G}{\cdot} \be$ and
$(-1)^{pq} \be \underset{G}{\cdot} \al$ is, of course, given by
Gerstenhaber's pre-Lie product \cite{Gerstenhaber}.  

For $\al \in W_p$ and $\be \in W_q$, the pairing $\langle \ , \
\rangle : W_* \ot W_* \to k$ is given by
\begin{align*}
 \langle \al, \, \be \rangle &  = \sum_{h, \, g_1, \, \ldots, \,
  g_{p+q} \in G} \al (h, \, g_1, \, \ldots \, , \, g_p) \be(h^{-1}, \,
g_{p+1}, \, \ldots \, , \, g_{p+q}) \\
& = \sum_{h \in G} \al(h, \, N, \, N, \, \ldots \, , \, N)
\be (h^{-1}, \, N, \, N, \, \ldots \, , \, N),
\end{align*}
where $N = \sum_{g \in G} g$ is the so-called norm element of $k[G]$.  
Clearly, the pairing is symmetric, i.e., $\langle \al, \, \be \rangle
= \langle \be, \, \al \rangle$.  

\begin{lemma}
For $\al \in W_p$, $\be \in W_q$ and $\gamma \in W_r$, we have
$$  \langle \al \underset{G}{\cdot} \be, \  \gamma \rangle =
\langle \al , \  \be \underset{G}{\cdot} \gamma \rangle. $$
\end{lemma}
\begin{proof}
Note that 
\begin{align*}
& \langle \al \underset{G}{\cdot} \be, \  \gamma \rangle = 
\sum_{\mu, \, h  \in G} ABC, \ \ \ 
A = \al (\mu, \, N, \, \ldots \, , \, N),  \\
& B = \be (h \mu^{-1}, \, N, \, \ldots \, , \, N), \ \ \ 
C = \gamma (h^{-1}, \, N, \, \ldots \, , \, N).
\end{align*}
Also,
\begin{align*} 
&  \langle \al , \  \be \underset{G}{\cdot} \gamma \rangle =
\sum_{\lambda, \, \ell, \in G} DEF, \ \ \
D = \al(\ell , \, N, \, \ldots \, , \, N), \\
&  E = \be (\lambda, \, N, \, \ldots \, , \, N), \ \ \
F = \gamma (\ell^{-1} \lambda^{-1}, \, N, \, \ldots \, , \, N).
\end{align*} 
Setting $\mu = \ell$ and $h = \lambda \mu$, we see that
$$  \langle \al \underset{G}{\cdot} \be, \  \gamma \rangle =
\langle \al , \  \be \underset{G}{\cdot} \gamma \rangle.  $$
\end{proof}

\begin{lemma}
The pairing $\langle \ , \ \rangle : W_* \ot W_* \to k$ induces
a well-defined map 
$$  (HH^p(k[G]), \, b^*) \ot (HH^q(k[G]), \, b^*) \to k  $$
for (i) $p$ odd, $q$ odd, (ii) $p$ even, $q$ odd, (iii) $p$ odd,
$q$ even, and (iv) $p = 0$, $q = 0$.  
\end{lemma}
\begin{proof}
Let $\al : k[G]^{\ot (p+1)} \to k$ and $\be : k[G]^{\ot (q+1)} \to k$
denote cocycles.  Let $\gamma : k[G]^{\ot p} \to k$ be an arbitrary
cochain.  In the group ring $k[G]$, $N^2  = {\nu} N$, where
$\nu$ is the order of $G$.  Now,
\begin{align*}
& \langle \al + b^*(\gamma), \, \be \rangle = \langle \al, \, \be
\rangle + \langle b^*(\gamma), \, \be \rangle, \\
& \langle b^*(\gamma), \, \be \rangle = \sum_{h \in G} b^*(\gamma) 
(h, \, N, \, \ldots \, , \, N) \be(h^{-1}, \, N, \, \ldots \, , \, N).
\end{align*}
For case (i),
\begin{align*}
b^*(\gamma)(h, \, N, \, \ldots \, , \, N)&  = \gamma(N, \, N, \, \ldots
\, , N) - \gamma(h, \,  N^2, \, \ldots \, , \, N) + \\
& \ldots + \gamma(h, \, N, \, \ldots \, , \, N^2) - \gamma(N, \, N, \,
\ldots \, , \, N) = 0.
\end{align*}
Thus, $\langle \al + b^*(\gamma), \, \be \rangle = \langle \al, \, \be
\rangle$.  Similarly in case (i), $\langle \al , \, \be + b^* (\theta)
\rangle = \langle \al , \, \be \rangle$, where $\theta : k[G]^{\ot q}
\to k$ is an arbitrary cochain.  

For case (ii),
$$  b^*(\gamma)(h, \, N, \, \ldots \, , \, N) = 
2 \gamma(N, \, N, \, \ldots \, ,\, N) - \nu \gamma(h,
\, N, \, \ldots \, \, N).  $$
Since $\be$ is a cocycle,
$$  0 = b^*(\be)(h, \, N, \, \ldots \, , \, N)  = 
2 \be (N, \, N, \, \ldots \, , \, N) - \nu \be (h, \, N, \, \ldots \,
, \, N).  $$
Thus,
\begin{align*}
\langle b^*(\gamma) , \, \be \rangle & = 
2 \gamma(N, \, N, \, \ldots \, , \, N) \be(N, \, N, \, \ldots \, , \,
N) \\
& - \sum_{h \in G} \nu \gamma(h, \, N, \, \ldots \, , \, N) 
\be (h^{-1}, \, N, \, \ldots \, , \, N) \\
& = 2 \gamma(N, \, N, \, \ldots \, , \, N) \be(N, \, N, \, \ldots \, , \,
N) \\
& -2 \sum_{h \in G} \gamma(h, \, N, \, \ldots \, , \, N) 
\be (N, \, N, \, \ldots \, , \, N) = 0.  
\end{align*}
Case (iii) follows by symmetry from case (ii).  Case (iv) is
well-known.  
\end{proof}

The pairing $\langle \ \, \rangle : W_* \ot W_* \to k$, however, is
not non-degenerate, since if $\al (h, \, N, \, \ldots \, , \, N) = 0$
for all $h \in G$, then $\langle \al , \, \be \rangle = 0$ for all
$\be \in W_*$.  In fact, it can be shown that given $\al \in W_p$ with
$\langle \al , \, \be \rangle = 0$ for all $\be$, then  
$\al (h, \, N, \, \ldots \, , \, N) = 0$ for all $h \in G$.  Let 
$$  V_p = \{ \al \in W_p \ | \  \al (h, \, N, \, \ldots \, , \, N) =
0 \ \forall h \in G \}.  $$
Then  $(V_*, \ b^*)$ is in fact a subcomplex of $(W_*, \ b^*)$.  With
the above restrictions, we still have:

\begin{corollary}
The cochain complex $(W_*, \ b^*)$ is a differential graded homotopy
commutative Frobenius algebra under the Gerstenhaber product and
pairing $\langle \ ,\ \rangle$ defined above.
\end{corollary}

The product $m: W_* \ot W_* \to W_*$ is a generalized convolution
product and $T^*: W_* \to W_* \ot W_*$ is the convolution coproduct.
Consequently any chain map $W_*^{\ot n} \to W_*^{\ot m}$ that is a
composition of the maps $m$ and  $T^*$ has a simple
expression in terms of generalized convolution products and
coproducts.   

For $n = 0$, 1, 2, $\ldots \, ,$ let
$$ N^{\rm{cy}}_n(G, \, e) = \{ (g_0, \, g_1, \, \ldots \, , \, g_n)
\in G^{n+1} \ | \ g_0g_1 \ldots g_n =e \}.  $$
Then $N^{\rm{cy}}(G, \, e)$ is a subsimplicial set of $N^{\rm{cy}}(G)$
with geometric realization $| N^{\rm{cy}}_*(G, \, e)| \simeq BG$, the
classifying space of $G$.  We say that 
$$ \al \in W_n = 
{\rm{Hom}}_k (k[G]^{\ot (n+1)}, \ k)  $$ 
is supported on $BG$ if
$\al (g_0, \, g_1, \, \ldots \, , \, g_n) = 0$ for 
$g_0 g_1 \ldots g_n \neq e$.  A direct calculation of the $b^*$
coboundary map shows that cocycles supported on $BG$ form a subcomplex
of the Hochschild complex $(W_*, \ b^*)$.  In symbols, let
$$  W_n(e) = \{ \al \in W_n \ | \ \al (g_0, \, g_1, \, \ldots \, , \,
g_n) = 0 \ \ {\rm{if}}\ \ g_0 g_1 \ldots g_n \neq e \},  $$
and let $W_*(e) = \sum_{n \geq 0} W_n(e)$ be the graded cochain complex
with $b^*$ as coboundary.  Then with coefficients in an arbitrary
commutative coefficient ring $k$, we have
$$  H^* (W_*(e); \, k) \simeq H^*(BG; \, k).  $$

\begin{lemma}
Under the Gerstenhaber product $W_*(e)$ becomes a differential graded
subalgebra of the full Hochschild cochain complex $(W_*, \ b^*)$.
\end{lemma}
\begin{proof}
The proof follows from a direct calculation using, for example,  
Lemma \eqref{Gerstenhaberproduct}, 
$$  \al \underset{G}{\cdot} \be =  
( \be_0 \al_0, \ \al_1, \ \ldots \, , \ \al_p, \ \be_1, \ \ldots \, , \
\be_q)^* , $$
where
$$  \al = (\al_0, \, \al_1, \, \ldots \, , \, \al_p)^*, \ \ \ 
\be = (\be_0, \, \be_1, \, \ldots \, , \, \be_q)^*.  $$
The result is also implicit in \cite{Farinati} and \cite{Menichi}.  
\end{proof}

Since $N^{\rm{cy}}_*(G)$ is a simplicial set, the cochain complex
$W_*$ is endowed with a second product, the simplicial cup product,
$\al \underset{S}{\cdot} \be$.  In particular, for $\al \in W_p$ and
$\be \in W_q$, the simplicial product $\al \underset{S}{\cdot} \be 
\in W_{p+q}$ is given by
\begin{align*}
&  (\al \underset{S}{\cdot} \be) (g_0, \, g_1, \, \ldots \, , \, g_p,
\, g_{p+1}, \, \ldots \, , \, g_{p+q}) = \\
&  \al ( (g_{p+1}g_{p+2} \ldots g_{p+q}g_0), \ g_1, \ g_2, \ \ldots \,
, \ g_p) \be ( (g_0g_1 \ldots g_p), \ g_{p+1}, \ \ldots \, , \ g_{p+q}),
\end{align*}
which is $\al$ evaluated on the front $p$-face and $\be$ evaluated on
the back $q$-face of $\sigma = (g_0, \, \ldots \, , \, g_{p+q})$.  
When restricted to $W_*(e)$, the product $\al \underset{S}{\cdot} \be$
realizes the usual cup product on the cochain complex for group
cohomology, $H^*(BG; \, k)$.  

\begin{lemma}
On the differential graded subalgebra $W_*(e)$, the Gerstenhaber
product is identical to the simplicial cup product, i.e., for $\al \in
W_p$ and $\be \in W_q$, we have $\al \underset{G}{\cdot} \be =
\al \underset{S}{\cdot} \be$.
\end{lemma}
\begin{proof}
This follows from a direct calculation \cite{Lodder}. 
\end{proof} 
The above cochain equality realizes that $H^*(BG; \, k)$ is a
subalgebra of $HH^*(k[G]; \, k[G])$, which is proven in 
\cite{Farinati} and \cite{Menichi}.   
More is true, namely $W_*(e)$ under the Gerstenhaber product is isomporphic to
$W_*(e)$ under the simplicial product as differential graded {\em homotopy
commutative} algebras.  For $\al \in W_p$ and $\be \in W_q$, let 
$\al \circ \be \in W_{p+q-1}$ denote the pre-Lie product of
Gerstenhaber \cite{Gerstenhaber}, which is a cochain homotopy between
$\al \underset{G}{\cdot} \be$ and $(-1)^{pq} \be \underset{G}{\cdot}
\al$.  Also let $\al \underset{1, \, S}{\cdot} \be \in W_{p+q-1}$
denote Steenrod's cup-one product \cite{Steenrod}, which is a cochain
homotopy between $\al \underset{S}{\cdot} \be$ and $(-1)^{pq} \be
\underset {S}{\cdot} \al$.   
 
\begin{lemma}
For $\al \in W_p (e)$ and $\be \in W_q (e)$, we have
$\al \circ \be = \al \underset{1, \, S}{\cdot} \be$ as cochains.
\end{lemma}
\begin{proof}
The proof follows from Theorem 3.11 and Corollary 3.12 of
\cite{Lodder}.
\end{proof}

We now show that the simplicial coproduct $\Delta : C_* \to C_* \ot
C_*$ agrees with the Gerstenhaber coproduct $T: C_* \to C_* \ot C_*$
when restricted to the subsimplicial complex $k[N^{\rm{cy}}_*(G, \,
e)]$ of $C_* = k[N^{\rm{cy}}_*(G)]$.  Recall that for 
$\sigma \in N^{\rm{cy}}_m(G) = G^{m+1}$,
$$  \Delta (\sigma) = \sum_{p=0}^m f_p(\sigma) \ot b_{m-p}(\sigma), $$
where $f_p(\sigma)$ denotes the front $p$-face of $\sigma$ and
$b_q(\sigma)$ denotes the back $q$-face of $\sigma$.  

\begin{lemma}
When restricted to $k[N^{\rm{cy}}_*(G, \, e)]$
$$  T: k[N^{\rm{cy}}_*(G, \, e)] \to k[N^{\rm{cy}}_*(G, \, e)] \ot
k[N^{\rm{cy}}_*(G, \, e)]  $$
agrees with
$$ \Delta : k[N^{\rm{cy}}_*(G, \, e)] \to C_* \ot C_* .  $$
\end{lemma}
\begin{proof}
For $\sigma \in  N^{\rm{cy}}_m (G, \, e)$, $\sigma = (g_0, \, g_1, \,
\ldots \, , \, g_m)$, we have
\begin{align*}
& T (\sigma) = e \ot \sigma + \\
& \ \sum_{p=1}^m \big( (g_1 g_2 \ldots g_p)^{-1}, \, g_1, \, \ldots \,
, \, g_p) \ot (g_0 g_1 \ldots g_p, \, g_{p+1}, \, \ldots \, , \, g_m).
\end{align*}
Also,
\begin{align*}
& \Delta (\sigma) = e \ot \sigma + \\
& \ \sum_{p=1}^m (g_{p+1} g_{p+2} \ldots g_m g_0, \, g_1, \, g_2, \,
\ldots \, , \, g_p) \ot (g_0g_1 \ldots g_p, \, g_{p+1}, \, \ldots \, ,
\, g_m).
\end{align*}
Since $g_0g_1 \ldots g_m = e$, it follows that $(g_1 g_2 \ldots
g_p)^{-1} = g_{p+1}g_{p+2} \ldots g_m g_0$.
\end{proof}

\end{document}